\documentclass[12pt]{article}
\usepackage{amsmath,amsthm,amsfonts,latexsym,amsopn,verbatim,amscd,amssymb}
\theoremstyle{plain}
\newtheorem{theorem}{Theorem}[section]
\newtheorem{lemma}[theorem]{Lemma}
\newtheorem{corollary}[theorem]{Corollary}
\newtheorem{definition}[theorem]{Definition}

\newcommand{\Z}{\mathbb{Z}}

\newcommand{\bnum}{\begin{enumerate}}
\newcommand{\enum}{\end{enumerate}}

\numberwithin{equation}{section}

\begin{document}

\title{\textbf{On commuting probability of finite rings}}
\author{Jutirekha Dutta, Dhiren Kumar Basnet and Rajat Kanti Nath\footnote{Corresponding author}}
\date{}
\maketitle
\begin{center}\small{\it 
Department of Mathematical Sciences, Tezpur University,\\ Napaam-784028, Sonitpur, Assam, India.\\



Emails:\, jutirekhadutta@yahoo.com,  dbasnet@tezu.ernet.in and rajatkantinath@yahoo.com}
\end{center}

\medskip

\begin{abstract} 
The commuting probability of a finite ring $R$, denoted by $\Pr(R)$, is the probability that any two randomly chosen elements of $R$ commute. In this paper, we obtain several bounds for $\Pr(R)$ through a generalization of $\Pr(R)$. Further, we define ${\Z}$-isoclinism between two pairs of rings and show that the generalized commuting probability, defined in this paper, is invariant under ${\Z}$-isoclinism between two pairs of finite rings.  
\end{abstract}

\medskip

\noindent {\small{\textit{Key words:}  finite ring, commuting probability, isoclinism of rings.}}  
 
\noindent {\small{\textit{2010 Mathematics Subject Classification:} 
16U70, 16U80.}} 

\medskip

\section{Introduction}

Throughout this paper $S$ denotes a subring of a finite ring  $R$.
We define $Z(S, R) = \{s \in S : sr = rs \; \forall \; r \in R\}$. Note that $Z(R, R) = Z(R)$, the center of $R$, and   $Z(S, R) = Z(R) \cap S$.
For any two elements $s$ and $r$ of a ring $R$, we write $[s, r]$ to denote the additive commutator of $s$ and $r$. That is, $[s, r] = sr - rs$. By $K(S, R)$ we denote the set $\{[s, r] : s \in S, r \in R\}$ and $[S, R]$ denotes the subgroup of $(R, +)$ generated by $K(S, R)$. Note that $[R, R]$ is the commutator subgroup of $(R, +)$ (see \cite{BMS}). Also, for any $x \in R$, we write $[x, R]$ to denote the subgroup of $(R, +)$ consisting of all elements of the form $[x, y]$ where $y \in R$.

The commuting probability of $R$, denoted by $\Pr(R)$, is the probability that a randomly chosen pair of elements of $R$ commute. That is 
\[
\Pr(R) = \frac{|\{(s, r) \in R \times R : sr = rs\}|}{|R \times R|}.
\]
The study of commuting probability of a finite ring was initiated  by MacHale \cite{dmachale} in the year 1976. Many papers have been written  on commuting probability of finite groups in the last few decades, for example see \cite{DasNath11,GR06,pL95,pL01,dmachale74,ND10,dR79} starting from the works of Erd$\ddot{\rm o}$s and Tur$\acute{\rm a}$n \cite{pEpT68}. However, the study of the commuting probability of a finite ring was neglected. After many years, in the year 2013, MacHale resumes the study of commuting probability of   finite rings together with Buckley and N$\acute{\rm i}$ Sh$\acute{\rm e}$ (see \cite{BM, BMS}).
In this paper, we obtain several bounds for $\Pr(R)$ through a generalization of $\Pr(R)$.  Motivated by \cite{erl,STM10,DN10} and \cite{nY15}, we generalize $\Pr(R)$ as the following ratio 
\begin{equation}\label{formula-001}
\Pr(S, R) = \dfrac{|\{(s, r) \in S \times R : sr = rs\}|}{|S \times R|}
\end{equation}
where $S$ is a subring of a finite ring $R$. Note that $\Pr(S, R)$ is the   
probability that a randomly chosen pair of elements, one from the subring $S$ and the other from $R$, commute.  
We call  $\Pr(S, R)$   the  relative commuting probability of the subring $S$ in the ring $R$. It is clear that  $\Pr(R, R) = \Pr(R)$ and $\Pr(S, R) = 1$ if and only if $Z(S, R) = S$. In Section 3 of this paper, we define ${\Z}$-isoclinism between two pairs of rings and show that  
$\Pr(S_1, R_1) = \Pr(S_2, R_2)$ if $(S_1, R_1)$ is ${\Z}$-isoclinic to $(S_2, R_2)$ where $S_1$ and $S_2$ are subrings of the rings $R_1$ and $R_2$ respectively.

   In this paper, we write   $R/S$ or $\frac{R}{S}$  to denote the additive quotient group, for any subring $S$ of $R$,  and $|R : S|$ to denote the index of    $(S, +)$ in   $(R, +)$. Further, if $S$ is  an ideal of $R$ then we also write  $R/S$ or $\frac{R}{S}$  to denote the quotient ring.    The isomorphisms considered are the additive group isomorphisms.   We shall also use the fact that for any non-commutative ring $R$, the additive group $\frac{R}{Z(R)}$ is not a cyclic group (see \cite[Lemma 1]{dmachale}).

\section{Some bounds}

Let $S$ be a subring of a ring $R$ and $r \in R$. We define a subring of $S$ given by  $C_S(r) = \{s \in S : sr = rs\}$. Then, from \eqref{formula-001}, it follows that
\begin{equation}\label{formula-1}
\Pr(S, R) = \frac{1}{|S||R|} \underset{s \in S}{\sum}|C_R(s)| = \frac{1}{|S||R|} \underset{r \in R}{\sum}|C_S(r)|.
\end{equation}
We also have the following lemma, which gives a relation between 
$|S : C_S(r)|$ and $|R : C_R(r)|$.  
\begin{lemma}\label{lemma1}
Let $S$ be a subring of a ring $R$ and $r \in R$. Then 
\[
|S : C_S(r)| \leq |R : C_R(r)|.
\]
The equality holds if and only if $S + C_R(r) = R$.
\end{lemma}
\begin{proof}
Let $r$  be any element of $R$. We know that $S + C_R(r) \subseteq R$, which gives $\frac{|S||C_R(r)|}{|S \cap C_R(r)|} \leq |R|$. Therefore, $\frac{|S||C_R(r)|}{|C_S(r)|} \leq |R|$    and hence the lemma follows.

For equality, it is sufficient to note that $|S : C_S(r)| = |R : C_R(r)|$ if and only if $|S + C_R(r)| = |R|$.  
\end{proof}
Above lemma plays an important role in finding   bounds for $\Pr(S, R)$ and hence for $\Pr(R)$. We begin with the following result which is an improvement of  \cite[Theorem 4]{dmachale}. 

\begin{theorem}\label{theorem01}
Let $S$ be a subring of a ring $R$. Then 
\[
\Pr(R) \leq \Pr(S, R) \leq \Pr(S).
\]
\end{theorem}
\begin{proof}
By \eqref{formula-1} and Lemma \ref{lemma1}, We have 
\[
\Pr(S, R)  = \frac{1}{|S| |R|}\underset{r \in R}{\sum}|C_S(r)|                    \geq  \frac{1}{|R| |R|} \underset{r \in R}{\sum}|C_R(r)| = \Pr(R)
\]   
and
\[
\Pr(S, R)  =   \frac{1}{|S| |R|} \underset{s \in S}{\sum}|C_R(s)|                    \leq  \frac{1}{|S| |S|} \underset{s \in S}{\sum}|C_S(s)| = \Pr(S).
\] 
Hence the theorem follows.
\end{proof}

\begin{corollary}\label{cor1}
Let $S$ be a subring of a ring $R$. Then 
\begin{enumerate}
 \item $\Pr(S, R) = \Pr(R)$ if and only if $S + C_R(r) = R$ for all $r \in R$.
\item $\Pr(S, R) = \Pr(S)$ if and only if $S + C_R(r) = R$ for all $r \in S$. 
\end{enumerate}
\end{corollary}
\begin{proof}
It is sufficient to note that  equalities hold in Theorem \ref{theorem01} if and only if    the equality holds in Lemma \ref{lemma1}.
\end{proof}
We remark that  $\Pr(S, R) = \Pr(R)$ implies $\Pr(S, R) = \Pr(S)$; but the converse  is not true, since  for a non-commutative ring $R$ and any subring $S \subseteq Z(R)$  we have $\Pr(S) = \Pr(S, R) = 1$. But $\Pr(R) \neq \Pr(S, R)$. There exist finite  rings and subrings such that the inequalities in Theorem \ref{theorem01} are strict. For example, consider the ring $R  = \left\lbrace  \begin{bmatrix}
    a & b\\
    0 & 0\\
  \end{bmatrix} : a, b \in \Z_2 \right\rbrace$ and its subring  $S = \left\lbrace  \begin{bmatrix}
    a & a\\
    0 & 0\\
  \end{bmatrix} : a \in \Z_2 \right\rbrace$. We have $\Pr(R) = \frac{5}{8},   \Pr(S, R) = \frac{3}{4}$ and  $\Pr(S) = 1$. Hence,  $\Pr(R) < \Pr(S, R) < \Pr(S)$.  

Further, we have the following result.
\begin{theorem}\label{refine}
Let $S_1$ and $S_2$ be two subrings of a ring $R$ such that $S_1 \subseteq S_2$. Then $\Pr(S_2, R)\leq \Pr(S_1, R) \leq \Pr(S_1, S_2)$. 
\end{theorem}

\begin{proof}
Since $S_1 \subseteq S_2 \subseteq R$ we have, by Lemma \ref{lemma1}, 
\[
|S_1 : C_{S_1}(r)| \leq |S_2 : C_{S_2}(r)| \leq |R : C_{R}(r)| \; \forall \; r \in R.
\]
Thus 
\begin{align*}
\Pr(S_1, S_2) = \frac{1}{|S_1||S_2|} \underset{s \in S_1}{\sum}|C_{S_2}(s)| \geq \frac{1}{|S_1|} \underset{s \in S_1}{\sum}\frac{|C_{S_2}(s)|}{|R|} = \Pr(S_1, R).
\end{align*}
Also 
\begin{align*}
\Pr(S_1, R) = \frac{1}{|S_1||R|} \underset{s \in R}{\sum}|C_{S_1}(s)| \geq \frac{1}{|R|} \underset{s \in R}{\sum}\frac{|C_{S_2}(s)|}{|S_2|} = \Pr(S_2, R).
\end{align*}
Hence the theorem follows.     
\end{proof}
Notice that $\Pr(R) \leq \Pr(S_2, R)$ and $\Pr(S_1, S_2) \leq \Pr(S_1)$. Therefore, the bounds obtained in Theorem \ref{refine} is a refinement of the bounds obtained in  Theorem \ref{theorem01}.

Let ${\mathcal{R}}_p$ denote the set of all finite rings having $p$ as the smallest prime dividing their orders. The next few  results give bounds for  $\Pr(S, R)$, where $S$ is a subring of a ring $R$ and $R$ is a ring in ${\mathcal{R}}_p$.

\begin{theorem}\label{theorem001}
If $R \in {\mathcal{R}}_p$ and $S$ is   a subring of $R$. Then 
\[
\frac{|Z(S, R)|}{|S|} + \frac{p(|S| - |Z(S, R)|)}{|S||R|} \leq \Pr(S, R) \leq \frac{(p -1)|Z(S, R)| + |S|}{p|S|}.
\]
\end{theorem}

\begin{proof}

We know that 
\begin{equation}\label{boundeq-3}
|S||R| \Pr(S, R) =  |R||Z(S, R)| + \underset{s \in S - Z(S, R)}{\sum}|C_R(s)|.
\end{equation}
 If $s \notin Z(S, R)$ then $p \leq |C_R(s)| \leq \frac{|R|}{p}$. Therefore 
\begin{equation}\label{boundeq-4}
\underset{s \in S - Z(S, R)}{\sum}p \leq \underset{s \in S - Z(S, R)}{\sum}|C_R(s)| \leq \underset{s \in S - Z(S, R)}{\sum}\frac{|R|}{p}.
\end{equation}
Hence, the result follows from \eqref{boundeq-3} and \eqref{boundeq-4}.                  
\end{proof}
Putting  $S = R$, in Theorem \ref{theorem001}, we get the following bounds for $\Pr(R)$.
\begin{corollary}\label{corprbd}
If $R \in {\mathcal{R}}_p$ and $S$ is   a subring of $R$. Then
\[
\frac{|Z(R)|}{|R|} + \frac{p(|R| - |Z(R)|)}{|R|^2} \leq \Pr(R) \leq \frac{(p - 1)|Z(R)| + |R|}{p|R|}. 
\]
\end{corollary}

If $R$ is a non-commutative ring and $p$  the smallest prime dividing $|R|$ then, by Theorem 2 of  \cite{dmachale}, we have $\Pr(R) \leq \frac{p^2 + p - 1}{p^3}$. For such ring $R$ we have $|R : Z(R)| \geq p^2$ and so
\[
\frac{(p - 1)|Z(R)| + |R|}{p|R|} \leq \frac{p^2 + p - 1}{p^3}.
\]
Thus the upper bound obtained in Corollary  \ref{corprbd} is better than the upper bound obtained in Theorem 2 of \cite{dmachale}.

  We also have the following bounds for $\Pr(S, R)$.
\begin{theorem}\label{theorem02}
Let $R \in {\mathcal{R}}_p$ and $S$ be   a subring of $R$.  
\bnum
\item If $S \nsubseteq Z(R)$ and $S$ is commutative, then $\Pr(S, R) \leq \frac{2p - 1}{p^2}$.
\item If $S \nsubseteq Z(R)$ and $S$ is non-commutative, then $\Pr(S, R) \leq \frac{p^2 + p - 1}{p^3}$.
\enum
\end{theorem}
\begin{proof}
(i) If $S \nsubseteq Z(R)$ then $|Z(S, R)| \leq \frac{|S|}{p}$. Hence, using Theorem \ref{theorem001}, we have 
\[
\Pr(S, R) \leq \frac{(p -1)|Z(S, R)| + |S|}{p|S|} \leq \frac{(p - 1)|S| + p|S|}{p^2|S|}   = \frac{2p - 1}{p^2}.
\]
\noindent (ii) If $S \nsubseteq Z(R)$ then $\Pr(S, R) \ne 1$.  Also, by Theorem \ref{theorem01}, we have  $\Pr(S, R) \leq \Pr(S)$. Since $S$ is non-commutative and $p$ is the smallest prime dividing $|S|$, the result follows from Theorem 2 of \cite{dmachale}.
\end{proof}
\noindent In particular, we have the following result.
\begin{corollary}\label{theorem2}
Let $R$ be any finite non-commutative ring and $S$   a subring of $R$. 
\bnum
\item If $S \nsubseteq Z(R)$ and $S$ is commutative, then $\Pr(S, R) \leq \frac{3}{4}$.
\item If $S \nsubseteq Z(R)$ and $S$ is not commutative, then $\Pr(S, R) \leq \frac{5}{8}$.
\enum
\end{corollary}

The following two results  characterize  a subring $S$ of a finite ring $R$ such that  $\Pr(S, R) = \frac{2p - 1}{p^2}$ or $\frac{p^2 + p - 1}{p^3}$.
\begin{theorem}\label{dc001}
Let $S$ be a commutative subring of a finite ring $R$ such that $\Pr(S, R) = \frac{2p - 1}{p^2}$, for some prime $p$. Then $p$ divides $|R|$. Moreover, if $p$ is the smallest prime dividing $|R|$ then
\[
\frac{S}{Z(S, R)} \cong {\mathbb{Z}}_p.
\] 
\end{theorem} 
\begin{proof}
The first part follows from the definition of $\Pr(S, R)$. For the second part, using Theorem \ref{theorem001}, we have $\Pr(S, R) \leq \frac{(p - 1)|Z(S, R)| + |S|}{p|S|}$ and so $|S : Z(S, R)| \leq p$, as $\Pr(S, R) = \frac{2p - 1}{p^2}$. If $|S : Z(S, R)| = 1$ then $\Pr(S, R) = 1$, which is a contradiction. Hence $|S : Z(S,\, R)| = p$ which gives $\frac{S}{Z(S,\, R)} \cong \Z_p$.
\end{proof}

\begin{theorem}\label{dc002}
Let $S$ be a non-commutative subring of a finite ring $R$ such that $\Pr(S, R) = \frac{p^2 + p - 1}{p^3}$, for some prime $p$. Then $p$ divides $|R|$. Moreover, if $p$ is the smallest prime dividing $|R|$ then
\[
\frac{S}{Z(S, R)} \cong {\mathbb{Z}}_p \times  {\mathbb{Z}}_p.
\] 
\end{theorem} 
\begin{proof}
The first part follows from the definition of $\Pr(S, R)$. For the second part, using Theorem \ref{theorem001}, we have $\Pr(S, R) \leq \frac{(p - 1)|Z(S, R)| + |S|}{p|S|}$ and so $|S : Z(S, R)| \leq p^2$, as $\Pr(S, R) = \frac{p^2 + p - 1}{p^3}$. Since $S$ is not-commutative,  $\frac{S}{Z(S, R)}$ is not cyclic. Therefore,  $|S : Z(S,\, R)| \ne 1, p$. So, $\frac{S}{Z(S, R)}$ is a non-cyclic group of order $p^2$. 
Hence   $\frac{S}{Z(S,\, R)} \cong {\mathbb{Z}}_p \times {\mathbb{Z}}_p$.
\end{proof}
\noindent In particular, for $p = 2$, we have the following results.
\begin{corollary}\label{dc}
Let $S$ be a subring of a finite ring $R$. 
\bnum
\item If $\Pr(S, R) = \frac{3}{4}$ and $S$ is commutative, then $\frac{S}{Z(S, R)} \cong \Z_2$.
\item If $\Pr(S, R) = \frac{5}{8}$ and $S$ is non-commutative, 
then $\frac{S}{Z(S, R)} \cong \Z_2 \times \Z_2$. 
\enum
\end{corollary}

In \cite{dmachale}, MacHale listed five results regarding commuting probability of finite groups. The ring theoretic analogue  of the first result of his list is proved in \cite{BMS}. Here we prove the ring theoretic analogue  of the last result. For this, we   prove the ring theoretic analogue  of  \cite[Theorem 3.9]{erl} from which the result follows. We begin with the following lemma.

\begin{lemma}\label{lemma2}
Let $H$ and $N$ be two subrings of a   non-commutative ring $R$ such that $N$ is an ideal of $R$ and $N \subseteq H$.
Then 
\[
\frac{C_H(x) + N}{N} \subseteq C_{H/N}(x + N) \; \text{for all} \; x \in R.
\]
The equality holds if $N \cap [H, R] = \{0\}$.
\end{lemma}

\begin{proof}
For any element $s \in C_H(x) + N$, where $s = r + n$ for some $r \in C_H(x)$ and $n \in N$, we have $s + N = r + N \in \frac{H}{N}$. Also, 
\[
(s + N)(x + N) = rx + N = xr + N = (x + N)(s + N), 
\]
as $r \in C_H(x)$. This proves the first part.

Let $N \cap [H, R] = \{0\}$ and $y + N \in C_{H/N}(x + N)$. Then $y \in H$ and  $(y + N)(x + N) = (x + N)(y + N)$. This gives $yx - xy \in N \cap [H, R] = \{0\}$ and so $y \in C_H(x)$. Therefore, $y + N \in \frac{C_H(x) + N}{N}$. Hence the equality holds. 
\end{proof}

\begin{theorem}\label{theorem3}
Let $H$ and $N$ be two  subrings of a finite non-commutative ring $R$ such that $N$ is an ideal of $R$ and $N \subseteq H$. Then 
\[
\Pr(H, R) \leq \Pr \left(\frac{H}{N}, \frac{R}{N} \right) \Pr(N).
\]
The equality holds if  $N \cap [H, R] = \{0\}$.
\end{theorem}

\begin{proof}
We have that 
\allowdisplaybreaks{
\begin{align*}
|H||R| \Pr(H, R) = &\underset{x \in R}{\sum} |C_H(x)| \\
= &\underset{S \in \frac{R}{N}}{\sum}\underset{y \in S}{\sum}\frac{|C_H(y)|}{|N \cap C_H(y)|} |C_N(y)| \\
= &\underset{S \in \frac{R}{N}}{\sum}\underset{y \in S}{\sum}\frac{|C_H(y) + N|}{|N|} |C_N(y)| \\
\leq &\underset{S \in \frac{R}{N}}{\sum}\underset{y \in S}{\sum}|C_{\frac{H}{N}}(y + N)| |C_N(y)| \;\;(\text{using Lemma \ref{lemma2}}) \\
= &\underset{S \in \frac{R}{N}}{\sum} |C_{\frac{H}{N}}(S)| \underset{y \in S}{\sum}|C_N(y)| \\
= &\underset{S \in \frac{R}{N}}{\sum} |C_{\frac{H}{N}}(S)| \underset{n \in N}{\sum}|C_R(n) \cap S|.  
\end{align*}
} 
Let $a + N = S$ where $a \in R - N$. If $C_R(n) \cap S = \phi$ then $|C_R(n) \cap S| < |C_N(n)|$. If $C_R(n) \cap S \neq \phi$ then there exists $x_0 \in C_R(n) \cap S$ such that $x_0 = a + n_0$ for some $a \in R - N$ and $n_0 \in N$. Therefore $x_0 + N = a + N = S$ and so $S \cap C_R(n) = (x_0 + N) \cap (x_0 + C_R(n)) = x_0 + (N \cap C_R(n)) = x_0 + C_N(n)$. Hence $|S \cap C_R(n)| \leq |C_N(n)|$. This gives
\begin{align*}
|H||R| \Pr(H, R) \leq &\underset{S \in \frac{R}{N}}{\sum} |C_{\frac{H}{N}}(S)| \underset{n \in N}{\sum}|C_N(n)| \\
= & \left|\frac{H}{N} \right| \left|\frac{R}{N} \right| \Pr \left(\frac{H}{N}, \frac{R}{N} \right) |N|^2 \Pr(N) \\
= & |H||R| \Pr \left(\frac{H}{N}, \frac{R}{N} \right) \Pr(N).  
\end{align*}
Hence the inequality follows.

 Let  $N \cap [H, R] = \{0\}$. Then, by Lemma \ref{lemma2}, we have  
 \[
\frac{C_H(x) + N}{N} = C_{H/N}(x + N) \; \text{for all} \; x \in R.
\]
If $S = a + N$ then it can be seen that $a + n \in C_R(n) \cap S$ for all $n \in N$. Therefore, $C_R(n) \cap S \ne \phi$ for all $n \in N$ and for all $S \in G/N$. Thus all the inequalities above become equalities if $N \cap [H, R] = \{0\}$. This completes the proof.
\end{proof}

Putting $H = R$, in Theorem \ref{theorem3},  we get the following corollary, which is analogous to the last result mentioned in \cite{dmachale}.

\begin{corollary}
Let $R$ be a finite non-commutative ring and $N$ be an ideal of $R$, then $\Pr(R) \leq \Pr(R/N) \Pr(N)$. The equality holds if $N \cap [R, R] = \{0\}$.
\end{corollary}

The following lemma is useful in  proving the next theorem.
\begin{lemma}[Observation 2.1 \cite{BMS}]\label{obs2.1}
Let $R$ be a finite ring. Then the additive group $R/C_R(x)$ is isomorphic to $[x, R]$ for all  $x \in R$.
\end{lemma}
\noindent Therefore, for all $s \in S$ we have 
\begin{equation}\label{eqlb}
|[S, R]| \geq |K(S, R)| \geq |[s, R]| = |R : C_R(s)|
\end{equation} 
\begin{theorem}\label{newlb1}
Let $S$ be a subring of a finite   ring  $R$. Then 
\[
\Pr(S, R) \geq \frac{1}{|K(S, R)|}\left(1 + \frac{|K(S, R)| - 1}{|S : Z(S, R)|} \right).
\]
In particular, if $Z(S, R) \ne S$ then $\Pr(S, R) > \frac{1}{|K(S, R)|}$.
\end{theorem}
\begin{proof}
By   \eqref{formula-1}, we have
\begin{align*}
\Pr(S, R) = & \frac{1}{|S|} \underset{s \in S}{\sum}\frac{1}{|R : C_R(s)|}\\
= & \frac{|Z(S, R)|}{|S|}   + \frac{1}{|S|}\underset{s \in S - Z(S, \,R)}{\sum}\frac{1}{|R : C_R(s)|}.
\end{align*}
Now, by \eqref{eqlb}, we have
\begin{align*}
\Pr(S, R) \geq & \frac{|Z(S, R)|}{|S|}   + \frac{1}{|S|} \underset{s \in S - Z(S, R)}{\sum}\frac{1}{|K(S, R)|}\\
= & \frac{|Z(S, R)|}{|S|}   +  \frac{|S| - |Z(S, R)|}{|S||K(S, R)|}
\end{align*}
from which the result follows.
\end{proof}

Since $|[S, R]| \geq |R : C_R(s)|$ for all $s \in S$, we also have the following lower bound.
\begin{theorem}\label{newlb2}
Let $S$ be a subring of a finite   ring  $R$. Then 
\[
\Pr(S, R) \geq \frac{1}{|[S, R]|}\left(1 + \frac{|[S, R]| - 1}{|S : Z(S, R)|} \right).
\]
In particular, if $Z(S, R) \ne S$ then $\Pr(S, R) > \frac{1}{|[S, R]|}$.
\end{theorem}

Let $p$ be the smallest prime dividing $|R|$. If $[S, R] \ne R$ and $S \ne Z(S, R)$ then it can be seen that
\[
\frac{1}{|[S, R]|}\left(1 + \frac{|[S, R]| - 1}{|S : Z(S, R)|} \right) \geq \frac{|Z(S, R)|}{|S|} + \frac{p(|S| - |Z(S, R)|)}{|S||R|}
\]
with equality  if and only if $|R : [S : R]| = p$. Also,
\[
\frac{1}{|K(S, R)|}\left(1 + \frac{|K(S, R)| - 1}{|S : Z(S, R)|} \right) \geq \frac{1}{|[S, R]|}\left(1 + \frac{|[S, R]| - 1}{|S : Z(S, R)|} \right) 
\]
with equality if and only if $K(S, R) = [S, R]$.
Hence, the lower bound obtained in Theorem \ref{newlb1} is better than the lower bounds obtained in   Corollary \ref{corprbd} and Theorem \ref{newlb2}.

Putting $S = R$, in   Theorem \ref{newlb1} and Theorem \ref{newlb2}, we get the following corollaries.
\begin{corollary}\label{newlb3}
If $R$ is a finite ring then 
\[
\Pr(R) \geq \frac{1}{|K(R, R)|}\left(1 + \frac{|K(R, R)| - 1}{|R : Z(R)|} \right).
\]
In particular, if $R$ is non-commutative then $\Pr(R) > \frac{1}{|K(R, R)|}$.
\end{corollary}
 
\begin{corollary}\label{newlb4}
If $R$ is a finite ring then 
\[
\Pr(R) \geq \frac{1}{|[R, R]|}\left(1 + \frac{|[R, R]| - 1}{|R : Z(R)|} \right).
\]
In particular, if $R$ is non-commutative then $\Pr(R) > \frac{1}{|[R, R]|}$.
\end{corollary}

We conclude this section, noting  that the lower bound for $\Pr(R)$ obtained in Corollary \ref{newlb3} is better than the lower bound obtained in Corollary \ref{corprbd} and Corollary \ref{newlb4}.
Further, the lower bound for $\Pr(R)$ obtained in Corollary \ref{newlb4}   is an improvement of the lower bound obtained in Lemma 2.3 of \cite{BMS}. Hence, the lower bound for $\Pr(R)$ obtained in Corollary \ref{newlb3} is better than the lower bound obtained in Lemma 2.3 of \cite{BMS}.

\section{$\Z$-isoclinism of rings}
Hall \cite{pH40} introduced the notion of isoclinism between two groups and Lescot \cite{pL95} showed that the commuting probability of two isoclinic finite groups are same. Later on Buckley, MacHale and N$\acute{\rm i}$ sh$\acute{\rm e}$ \cite{BMS} introduced the concept of $\Z$-isoclinism between two rings and showed that the commuting probability of two isoclinic finite rings are same. In this section, we introduce the concept of $\Z$-isoclinism between two pairs of rings and show that relative commuting probability remains invariant under $\Z$-isoclinism of pairs of rings. The group theoretic analogous results can be found in \cite{nY15}.

 Recall that two rings $R_1$ and $R_2$ are said to be $\Z$-isoclinic (see \cite{BMS})  if there exist additive group isomorphisms $\phi : \frac{R_1}{Z(R_1)} \rightarrow \frac{R_2}{Z(R_2)}$   and $\psi : [R_1, R_1] \rightarrow [R_2, R_2]$ such that $\psi ([u, v]) = [u', v']$ whenever $\phi(u + Z(R_1)) = u' + Z(R_2)$ and $\phi(v + Z(R_1)) = v' + Z(R_2)$.  Equivalently, the following diagram commutes 
\vspace{.25cm}
\begin{center}
$
\begin{CD}
   \frac{R_1}{Z(R_1)} \otimes \frac{R_1}{Z(R_1)} @>\phi \otimes \phi>> \frac{R_2}{Z(R_2)} \otimes \frac{R_2}{Z(R_2)}\\
   @VV{a_{R_1}}V  @VV{a_{R_2}}V\\
   [R_1, R_1] @>\psi>> [R_2, R_2],
\end{CD}
$
\end{center}
where $\frac{R_i}{Z(R_i)} \otimes \frac{R_i}{Z(R_i)}$ denotes  tensor product of $\frac{R_i}{Z(R_i)}$ with itself for $i = 1, 2$; $a_{R_i}: \frac{R_i}{Z(R_i)} \otimes \frac{R_i}{Z(R_i)} \to [R_i, R_i]$  are well defined maps given by 
\[
a_{R_i}((x_i + Z(R_i))\otimes (y_i + Z(R_i))) = [x_i, y_i]
\]
 for all $x_i, y_i \in R_i$ and $i = 1, 2$; and  
\[
\phi \otimes \phi((x_1 + Z(R_1)\otimes (y_1 + Z(R_1)) = (x_2 + Z(R_2)\otimes (y_2 + Z(R_2)
\]  
whenever
$\phi(x_1 + Z(R_1)) = x_2 + Z(R_2)$ and $\phi(y_1 + Z(R_1)) = y_2 + Z(R_2)$.
The above diagram commutes means
\[
a_{R_2} \circ (\phi \otimes \phi) = \psi \circ a_{R_1}.
\]
The pair of mappings $(\phi, \psi)$ is called a $\Z$-isoclinism from $R_1$ to $R_2$.
\vspace{.25cm}

Following \cite{nY15}, we  introduce the concept of $\Z$-isoclinism between two pairs of rings in the following definition.

\begin{definition}
Let $R_1$ and $R_2$ be two rings with subrings $S_1$ and $S_2$ respectively. A pair of rings $(S_1, R_1)$ is said to be $\Z$-isoclinic to a pair of rings $(S_2, R_2)$ if there exist additive group isomorphisms $\phi : \frac{R_1}{Z(S_1, R_1)} \rightarrow \frac{R_2}{Z(S_2, R_2)}$ such that $\phi \left(\frac{S_1}{Z(S_1, R_1)}\right) = \frac{S_2}{Z(S_2, R_2)}$; and $\psi : [S_1, R_1] \rightarrow [S_2, R_2]$ such that $\psi ([u, v]) = [u', v']$ whenever $\phi(u + Z(S_1, R_1)) = u' + Z(S_2, R_2)$ and $\phi(v + Z(S_1, R_1)) = v' + Z(S_2, R_2)$. Such pair of mappings $(\phi, \psi)$ is called a $\Z$-isoclinism between $(S_1, R_1)$ and $(S_2, R_2)$.
\end{definition}
Equivalently, a pair of rings $(S_1, R_1)$ is said to be $\Z$-isoclinic to a pair of rings $(S_2, R_2)$ if there exist additive group isomorphisms $\phi : \frac{R_1}{Z(S_1, R_1)} \rightarrow \frac{R_2}{Z(S_2, R_2)}$ such that $\phi \left(\frac{S_1}{Z(S_1, R_1)}\right) = \frac{S_2}{Z(S_2, R_2)}$ and  $\psi : [S_1, R_1] \rightarrow [S_2, R_2]$ such that the following diagram commutes 
\vspace{.25cm}
\begin{center}
$
\begin{CD}
   \frac{S_1}{Z(S_1, R_1)} \otimes \frac{R_1}{Z(S_1, R_1)} @>\phi \otimes \phi>> \frac{S_2}{Z(S_2, R_2)} \otimes \frac{R_2}{Z(S_2, R_2)}\\
   @VV{a_{(S_1, R_1)}}V  @VV{a_{(S_2, R_2)}}V\\
   [S_1, R_1] @>\psi>> [S_2, R_2],
\end{CD}
$
\end{center}
where $\frac{S_i}{Z(S_i, R_i)} \otimes \frac{R_i}{Z(S_i, R_i)}$ denotes tensor product of $\frac{S_i}{Z(S_i, R_i)}$ with $\frac{R_i}{Z(S_i, R_i)}$ for $i = 1, 2$; $a_{(S_i, R_i)}: \frac{S_i}{Z(S_i, R_i)} \otimes \frac{R_i}{Z(S_i, R_i)} \to [S_i, R_i]$  are well defined maps given by 
\[
a_{(S_i, R_i)}((x_i + Z(S_i, R_i))\otimes (y_i + Z(S_i, R_i))) = [x_i, y_i]
\]
 for all $x_i \in S_i, y_i \in R_i$ and  $i = 1, 2$;  and
\[
\phi \otimes \phi((x_1 + Z(S_1, R_1))\otimes (y_1 + Z(S_1, R_1))) = (x_2 + Z(S_2, R_2))\otimes (y_2 + Z(S_2, R_2))
\]  
whenever
$\phi(x_1 + Z(S_1, R_1)) = x_2 + Z(S_2, R_2)$ and $\phi(y_1 + Z(S_1, R_1)) = y_2 + Z(S_2, R_2)$. The above diagram commutes means
\[
a_{(S_2, R_2)} \circ (\phi \otimes \phi) = \psi \circ a_{(S_1, R_1)}.
\]

\noindent For example, consider the non-commutative rings
\[
R_1 =\left\lbrace  \begin{bmatrix}
    a & b\\
    0 & c\\
  \end{bmatrix} : a, b, c \in \Z_p \right\rbrace, 
 R_2 =\left\lbrace  \begin{bmatrix}
    x & y\\
    0 & 0\\
  \end{bmatrix} : x, y \in \Z_p \right\rbrace
\]
and their subrings $S_1 = \left\lbrace  \begin{bmatrix}
    a & 0\\
    0 & a\\
  \end{bmatrix} : a \in \Z_p \right\rbrace$,
 $S_2 =   \left\lbrace  \begin{bmatrix}
    0 & 0\\
    0 & 0\\
  \end{bmatrix}\right\rbrace$ respectively. Then  $R_1/Z(S_1, R_1) \cong \Z_p \times \Z_p \cong R_2/Z(S_2, R_2)$ and $[S_1, R_1] \cong \{0\} \cong [S_2, R_2]$. Also the above diagram commutes since $[S_1, R_1]$ and $[S_2, R_2]$ are trivial. Hence the pairs $(S_1, R_1)$ and $(S_2, R_2)$ are $\Z$-isoclinic. But the pairs  $(\langle 4 \rangle, \Z_8)$ and $(\langle 3 \rangle, \Z_{12})$ are not $\Z$-isoclinic  since $\frac{\Z_8}{Z(\langle 4 \rangle, \Z_8)} \cong \Z_4$ but $\frac{\Z_{12}}{Z(\langle 3 \rangle, \Z_{12})}  \cong \Z_3$. 

The following lemma plays an important role in proving the invariance property of relative commuting probability under  $\Z$-isoclinism between two pairs of rings.

\begin{lemma}\label{isolem}
Let $R_1$ and $R_2$ be two rings with subrings $S_1$ and $S_2$ respectively. If $(\phi, \psi)$ is a $\Z$-isoclinism from  $(S_1, R_1)$ to   $(S_2, R_2)$ then $[s_1, R_1]$ and $[s_2, R_2]$ are isomorphic, whenever $s_1 \in S_1, s_2 \in S_2$ and $\phi(s_1 + Z(S_1, R_1)) = s_2 + Z(S_2, R_2)$.
\end{lemma}
\begin{proof}
If $(\phi, \psi)$ is a $\Z$-isoclinism from  $(S_1, R_1)$ to   $(S_2, R_2)$ then $\phi : \frac{R_1}{Z(S_1, R_1)} \rightarrow \frac{R_2}{Z(S_2, R_2)}$ and $\psi : [S_1, R_1] \rightarrow [S_2, R_2]$ are additive group isomorphisms such that $\phi \left(\frac{S_1}{Z(S_1, R_1)}\right) = \frac{S_2}{Z(S_2, R_2)}$. Let $s_1 \in S_1$ and $\psi'$ denote the restriction of $\psi$ on $[s_1, R_1]$. Then $\psi'$ is an injective homomorphism, since $\psi$ is an isomorphism. Let $s_2 \in S_2$ such that $\phi(s_1 + Z(S_1, R_1)) = s_2 + Z(S_2, R_2)$.
We shall show that $\psi'$ is the required isomorphism between $[s_1, R_1]$ and $[s_2, R_2]$. For this it is sufficient to show that $\psi'$ is surjective. Let $[s_2, r_2] \in [s_2, R_2]$ for some $r_2\in R_2$. Since    
$\phi : \frac{R_1}{Z(S_1, R_1)} \rightarrow \frac{R_2}{Z(S_2, R_2)}$ is an isomorphism, there exists an element $r_1 \in R_1$ such that $\phi(r_1 + Z(S_1, R_1)) = r_2 + Z(S_2, R_2)$. 
Since $(\phi, \psi)$ is a $\Z$-isoclinism,  we have 
\begin{align*}
\psi'([s_1, r_1]) = & \psi([s_1, r_1]) = \psi \circ a_{(S_1, R_1)} ((s_1 + Z(S_1, R_1))\otimes (r_1 + Z(S_1, R_1))) \\
= & a_{(S_2, R_2)}\circ(\phi \otimes \phi)((s_1 + Z(S_1, R_1))\otimes (r_1 + Z(S_1, R_1))) \\
= & a_{(S_2, R_2)}((s_2 + Z(S_2, R_2))\otimes (r_2 + Z(S_2, R_2)))\\
= & [s_2, r_2].
\end{align*}
This shows that $\psi'$ is  surjective. Hence the lemma follows.
\end{proof}
 
Now, we prove the main theorem of this section. 
\begin{theorem}
Let $R_1$ and $R_2$ be two rings with subrings $S_1$ and $S_2$ respectively. If $(\phi, \psi)$ is a $\Z$-isoclinism from   $(S_1, R_1)$ to   $(S_2, R_2)$ then 
\[
\Pr(S_1, R_1) = \Pr(S_2, R_2).
\]
\end{theorem}

\begin{proof}
 By \eqref{formula-1} and Lemma \ref{obs2.1}, we have
\[
\Pr(S_1, R_1) = \frac{1}{|S_1 : Z(S_1, R_1)|} \underset{s_1 + Z(S_1, R_1) \in \frac{S_1}{Z(S_1, R_1)}}{\sum}\frac{1}{|[s_1, R_1]|} 
\]
and 
\[
\Pr(S_2, R_2) = \frac{1}{|S_2 : Z(S_2, R_2)|} \underset{s_2 + Z(S_2, R_2) \in \frac{S_2}{Z(S_2, R_2)}}{\sum}\frac{1}{|[s_2, R_2]|}.
\]
If $(\phi, \psi)$ is a $\Z$-isoclinism from   $(S_1, R_1)$ to   $(S_2, R_2)$ then $|S_1 : Z(S_1, R_1)| = |S_2 : Z(S_2, R_2)|$. Also, by Lemma \ref{isolem}, we have 
\[
\underset{s_1 + Z(S_1, R_1) \in \frac{S_1}{Z(S_1, R_1)}}{\sum}\frac{1}{|[s_1, R_1]|} = \underset{s_2 + Z(S_2, R_2) \in \frac{S_2}{Z(S_2, R_2)}}{\sum}\frac{1}{|[s_2, R_2]|}.
\]
Hence the theorem follows.    
\end{proof}





\end{document}